\documentclass[12pt]{amsart}
\usepackage{amsfonts}
\usepackage{mathrsfs}
\usepackage{amscd,amsmath,amssymb,amsfonts}
\usepackage[all]{xy}
\theoremstyle{plain}
\newtheorem{thm}{Theorem}
\newtheorem{lem}[thm]{Lemma}

\newtheorem{prop}[thm]{Proposition}

\theoremstyle{definition}
\newtheorem{defn}[thm]{Definition}

\newtheorem{rmks}[thm]{Remarks}

\numberwithin{thm}{section} \numberwithin{equation}{section}

\newcommand{\ga}[2]{\begin{gather}\label{#1}#2 \end{gather}}

\newcommand{\sD}{{\mathcal D}}
\newcommand{\sE}{{\mathcal E}}
\newcommand{\sF}{{\mathcal F}}

\newcommand{\sL}{{\mathcal L}}

\newcommand{\sN}{{\mathcal N}}
\newcommand{\sO}{{\mathcal O}}

\newcommand{\sU}{{\mathcal U}}

\newcommand{\sX}{{\mathcal X}}

\newcommand{\sZ}{{\mathcal Z}}

\newcommand{\F}{{\mathbb F}}

\begin{document}

\title{Stratified bundles and Representation spaces}

\author{Xiaotao Sun}
\address{Center of Applied Mathematics, School of Mathematics, Tianjin University, No.92 Weijin Road, Tianjin 300072, P. R. China}
\email{xiaotaosun@tju.edu.cn}

\date{July 6, 2017}
\thanks{}
\begin{abstract} For a given stratified bundle $E$ on $X$, we construct an irreducible closed subvariety $\sN(E)_S$ of the so called representation space
$R(\sO_{X_S},\xi_S,P)\to S$ such that $\sN(E)_S(\overline{\mathbb{F}}_q)$ contains a dense set of $(V,\beta)$ where $V$ is induced by a representation of $\pi_1^{{\rm \acute{e}t}}(X)$ (Theorem \ref{thm3.7}). As an application, we give a simply proof of the main theorem of \cite{EM} and its relative version (Theorem \ref{thm4.2}).
\end{abstract}
\maketitle

\section{Introduction}

Let $X$ be a smooth, connected projective variety over an algebraically closed field $k$ of characteristic $p>0$, $\sD_X$ be the sheaf of differential
operators (in the sense of Grothendieck) and $\pi_1=\pi_1^{\text{{\'e}t}}(X, \xi)$ be the {\'e}tale fundamental group of $X$. For any representation $\rho:\pi_1\to {\rm GL}(V)$, one can associate to $\rho$ a $\sD_X$-module $V_{\rho}$. Thus D. Gieseker proved the following results (see Theorem 1.10 of \cite{Gi}): (i) if every $\sD_X$-module on $X$ is trivial, then $\pi_1$ is trivial; (ii) if all irreducible $\sD_X$-modules are rank $1$, then $[\pi_1,\pi_1]$ is a pro-$p$-group;
(iii) if every $\sD_X$-module is a direct sum of rank $1$ $\sD_X$-modules, then $\pi_1$ is abelian with no $p$-power order quotient. Following D. Gieseker,
a $\sD_X$-module $E$ will be called a stratified bundle.

Gieseker also made the conjecture that the converses of above statements might be true. The converse of statement (i) was proved in \cite{EM}, and converses of the statements (ii) and (iii) were proved in \cite{ES}.
The key in these proofs is to produce a representation of $\pi_1=\pi_1^{\text{{\'e}t}}(X, \xi)$ from a given stratified bundle $E$. An equivalent characterization of stratified bundle is that $E=(E_i)_{i\in\mathbb{N}}$ with $E_i=F_X^*E_{i+1}$ ($\forall\,\,i\in\mathbb{N}$) where $F_X:X\to X$ is the Frobenius map, and there is an integer $n_0$ such that
$E_i$ ($i\ge n_0$) are $p$-semistable bundles with trivial chern classes.

If $\Sigma=\{E_i\}_{i\ge n_0}$ is finite, then there is an $F$-periodic bundle $E_{i_0}$ (i.e. there is an integer $N$ such that $(F_X^*)^NE_{i_0}=E_{i_0}$) which induces a representation of $\pi_1$ by a theorem of Lange-Stuhler (Lemma \ref{lem3.4}).

When $\Sigma=\{E_i\}_{i\ge n_0}$ is an infinite set, a theorem of Hrushovski is used to get an $F$-periodic bundle on a good reduction $X_{\bar s}/\overline{\mathbb{F}}_q$ of $X$,
which says that a dominant rational map $f: Y\dashrightarrow Y$ of varieties over $\overline{\mathbb{F}}_q$ has a dense set of $f$-periodic points (Lemma \ref{lem3.6}). If we have a moduli space $M$ parametrizing \textbf{isomorphism classes} of semistable bundles, we would have a subvariety $\sN(E)\subset M$ (by taking Zariski closure of $\Sigma=\{E_i\}_{i\ge n_0}$) such that Frobenius pullback $F_X^*$ induces a dominant rational map $F_X^*:\sN(E)\dashrightarrow \sN(E)$. Then, if $k=\overline{\mathbb{F}}_q$, we
find a dense set of $F$-periodic bundles (thus a dense set of representations of $\pi_1$) by Hrushovski's theorem. Unfortunately, we have only a moduli space $M$ parametrizing
\textbf{$s$-equivalence classes} of semistable bundles. Thus the approach in \cite{EM} consists of two steps: (1) prove the theorem for irreducible stratified bundles (in this case, $\Sigma=\{E_i\}_{i\ge n_0}$ consists of stable bundles), (2) studying the extensions of irreducible stratified bundles, which sometimes is rather involved.

Let $X$ be a projective variety over a perfect field $k$ with a point $$\xi:{\rm Spec}(k)\to X.$$
We observe in this article that for any stratified bundle $E=(E_i)_{i\in \mathbb{N}}$ of rank $r$ there is a natural way to choose frames $\beta_i:\xi^*E_i\cong \sO_X^{\oplus r}$ such that
$(E_i,\beta_i)=F^*_X(E_{i+1},\beta_{i+1})$
(see Lemma \ref{lem3.3}). Moreover, the set $R(E)_{n_0}=\{\alpha_i=(E_i,\beta_i)\}_{i\ge n_0}$ is a set of $k$-points of a moduli space $R(\sO_X,\xi, P)$, which parametrizes
\textbf{isomorphism classes} of $(V,\beta)$ (i.e. semistable bundles $V$ with frames $\beta$ at $\xi\in X$) and was called the \textbf{Representation Space} by Simpson.

In Section 2 of this article, we generalize Simpson's construction of representation spaces $R(\sO_X,\xi,P)$ to the case of characteristic $p>0$ (see Theorem \ref{thm2.3})
and prove that Frobenius pullback $F^*_X$ induces a rational map $f:R(\sO_X,\xi,P)\dashrightarrow R(\sO_X,\xi,P)$ (see Proposition \ref{prop2.5}). In Section 3, for a
stratified bundle $E=(E_i)_{i\in\mathbf{N}}$ such that $\Sigma=\{E_i\}_{i\ge n_0}$ is an infinite set, we construct a closed subvariety $\sN(E)\subset R(\sO_X,\xi,P)$
such that $f:R(\sO_X,\xi,P)\dashrightarrow R(\sO_X,\xi,P)$ induces a dominant rational map $f^a:\sN(E)\dashrightarrow \sN(E)$ and $\sN(E)(k)\cap R(E)_{n_0}$ is an infinite
set (see Theorem \ref{thm3.7}). In Section 4, we use the construction of Section 3 to give a uniform proof (see Theorem \ref{thm4.1}) of the main theorem in \cite{EM}, which says that there is no nontrivial stratified bundle on $X$ if $\pi_1=\pi_1^{\text{{\'e}t}}(X, \xi)$ is trivial. For example, when $k=\overline{\mathbb{F}}_q$, $\sN(E)$ contains
a dense set of points $(V,\beta)$ such that $V$ is induced by a representation of $\pi_1$. On the other hand, if $E=(E_i)_{i\in \mathbf{N}}$ is nontrivial, we can assume that
all bundles $E_i$ in $\Sigma=\{E_i\}_{i\ge n_0}$ are nontrivial, then the set $$\sU=\{\,(V,\beta)\in \sN(E)\,|\,\text{$V$ is nontrivial}\}$$ is a nonempty open set, which must
contain a point $(V,\beta)$ such that $V$ is induced by a representation of $\pi_1$ and we get a contradiction if $\pi_1$ is trivial. These arguments are easily applied to prove
relative version of this theorem (see Theorem \ref{thm4.2}).\\[.2cm]
{\it Acknowledegements:} Theorem \ref{thm4.2} (see \cite{ESr} for an another proof) was a question that H{\'e}l{\`e}ne Esnault posed to me
when I visited Berlin on 2013, where I proved immediately the irreducible case of Theorem \ref{thm4.2} in a unpublished note (in fact, I proved the theorem
for stratified bundles which are extensions of two irreducible stratified bundles). I thank her very much for the question and discussions.

\section{Representation spaces and Frobenius map}

Let $X$ be a nonsingular projective variety over a perfect field $k$, fix an ample line bundle $\sO_X(1)$ on $X$. For a torsion free sheaf $\sE$ of rank $r(\sE)$ on $X$,
$P(\sE, m)=\chi(\sE(m))$ is a polynomial in $m$ (the so called Hilbert polynomial of $\sE$) with degree $n={\rm dim}\,X$.

A torsion free sheaf $\sE$ on $X$ is called $p$-semistable (resp. $p$-stable) if for any subsheaf $\sF\subset \sE$, when $m$ large enough, we have
$$p(\sF,m):=\frac{P(\sF,m)}{r(\sF)}\le \frac{P(\sE,m)}{r(\sE)}:=p(\sE,m)\quad ({\rm resp.}\,\,<\,).$$

\begin{lem}\label{lem2.1} Let $0\to \sE_1\to\sE\to\sE_2\to 0$ be an exact sequence of torsion free sheaves,
if $\sE_1$ and $\sE_2$ are $p$-semistable with $p(\sE_1,m)=p(\sE_2,m)$, then $\sE$ is $p$-semistable with $p(\sE,m)=p(\sE_1,m)=p(\sE_2,m)$.
\end{lem}

\begin{proof} It is easy to check that $p(\sE,m)=p(\sE_1,m)=p(\sE_2,m)$. For a subsheaf $\sF\subset \sE$, let $\sF_2\subset\sE_2$ be the
image of $\sF$ under the surjection $\sE\to\sE_2\to 0$ and $\sF_1\subset \sE_1$ be the subsheaf such that
$$0\to \sF_1\to\sF\to\sF_2\to 0$$ is exact. If $\sF_1=0$ (resp. $\sF_2=0$), we have $$p(\sF,m)=p(\sF_2,m)\le p(\sE_2,m)=p(\sE,m)$$
(resp. $p(\sF,m)=p(\sF_1,m)\le p(\sE_1,m)=p(\sE,m)$). Thus we assume that both $\sF_1$ and $\sF_2$ are nontrivial. Then we have
$$\aligned p(\sF,m)=&\frac{r(\sF_1)p(\sF_1,m)+r(\sF_2)p(\sF_2,m)}{r(\sF)}\\&\le\frac{r(\sF_1)p(\sE_1,m)+r(\sF_2)p(\sE_2,m)}{r(\sF)}=p(\sE,m).\endaligned$$

\end{proof}

Let $S$ be an affine variety over a finite field $\mathbb{F}_q$, and $X_S\to S$ be a projective, flat morphism
with geometrically irreducible fibers. Fix a polynomial $P$ of degree equal to the relative
dimension $d=dim(X_S/S)$ and a relative ample line bundle $\sO_{X_S}(1)$ on $X_S$. Let
$${\rm Quot}_P(\sO_{X_S}(-N)^{\oplus P(N)})\to S $$
be the relative quotient scheme
parametrizing quotients $$\sO_{X_s}(-N)^{\oplus P(N)}\to\sF_s\to 0$$ on geometric fibers $X_s$ of $X_S\to S$ with Hilbert polynomial $P$. It is known that there
is an ample line bundle $\sL_m$ on ${\rm Quot}_P(\sO_{X_S}(-N)^{\oplus P(N)})$ such that the open set $Q\subset {\rm Quot}_P(\sO_{X_S}(-N)^{\oplus P(N)})$ of GIT semistable (resp. GIT stable) points
under the action of ${\rm GL}(P(N))$ (respect to $\sL_m$) is precisely the open set of quotients $\sO_{X_s}(-N)^{\oplus P(N)}\to\sF_s\to 0$ where $\sF_s$ are $p$-semistable (resp. $p$-stable) torsion free sheaves on $X_s$ (See \cite{Si} or \cite{La} for positive characteristic). Let
\ga{2.1} {\varphi: Q\to M(\sO_{X_S},P):=Q//{\rm GL}(P(N))}
be the GIT quotient over $S$ defined in Theorem 4 of \cite{Se}. Then $$M(\sO_{X_S},P)\to S$$ is a projective scheme of finite type over $S$, which uniformly corepresents the functor
$\mathbf{M}(\sO_{X_S},P): {\rm Sch}/S\to {\rm Sets}$ defined by
$$\mathbf{M}(\sO_{X_S},P)(S')=\left\{\aligned&\text{$s$-equivalence classes of families of $p$-semistable}\\& \text{ sheaves on the geometric fibres of $X_{S'}\to S\,'$,}
\\& \text{which are flat over $S'$ with Hilbert polynomial $P$}
\endaligned \right\}.$$

\begin{defn}\label{defn2.2} Suppose $\xi:S\to X_S$ is a section of $X_S\to S$ and $\sF$ is a $p$-semistable sheaf with Hilbert polynomial $P$. Then $\sF$ is called
satisfying condition ${\rm LF}(\xi)$ if $gr(\sF_s)$ is loaclly free at $\xi(s)$ ($\forall\,\,s\in S$).
\end{defn}

Let $Q^{{\rm LF}(\xi)}\subset Q $ be the subset of $Q$ parametrizing quotients $$\sO_{X_S}(-N)^{\oplus P(N)}\to\sF\to 0$$ where $\sF$ satisfies condition ${\rm LF}(\xi)$. It was shown in \cite{Si} that
there is an open set $M^{{\rm LF}(\xi)}(\sO_{X_S},P)\subset M(\sO_{X_S},P)$ such that $$Q^{{\rm LF}(\xi)}=\varphi^{-1}(M^{{\rm LF}(\xi)}(\sO_{X_S},P))$$ and
$\varphi: Q^{{\rm LF}(\xi)}\to M^{{\rm LF}(\xi)}(\sO_{X_S},P)$ is an uniform categorical quotient.

Let $\sF^{univ}$ be the universal quotient on $X_S\times_SQ^{{\rm LF}(\xi)}$, which is locally free along the universal section
$\xi:Q^{{\rm LF}(\xi)}\to X_S\times_S Q^{{\rm LF}(\xi)}$, and let $$\pi: T\to Q^{{\rm LF}(\xi)}$$ be the frame bundle of $\xi^*(\sF^{univ})$, which
represents the functor that associates to any $S'\to S$ the set of all triples $(\sE, \alpha, \beta)$, where $\sE$ is a $p$-semistable torsion free sheaf of
Hilbert polynomial $P$ on $X_{S'}/S'$ satisfying condition ${\rm LF}(\xi)$, and $\alpha$, $\beta$ are isomorphisms
$$\alpha: \sO_{S'}^{\oplus P(N)}\cong H^0(X_{S'}/S',\sE(N)), \quad \beta:\xi^*(\sE)\cong \sO_{S'}^{\oplus r}.$$
The group ${\rm GL}(P(N))\times {\rm GL}(r)$ acts on $T$, compatibly with the action of ${\rm GL}(P(N))$ on $ Q^{{\rm LF}(\xi)}$. We may choose a linearization of the action of
${\rm GL}(P(N))$ on $\sL_m^b$ such that the center $G_m\subset {\rm GL}(P(N))$ acts trivially. Then the line bundle $\sL_m^b$ on $Q^{{\rm LF}(\xi)}$ has a linearization with respect
to the group ${\rm GL}(P(N))\times {\rm GL}(r)$, where the second factor acts trivially. Let $\textbf{L}$ denote the pullback of the ${\rm GL}(P(N))\times {\rm GL}(r)$-linearized bundle $\sL_m^b$ to $T$. Then we have a characteristic $p$ analogue of a special case of Simpson's result (see Theorem 4.10 of \cite{Si}).

\begin{thm}\label{thm2.3} Every point of $\,T$ is stable for the action of ${\rm GL}(P(N))$ with respect to the linearized line bundle $\textbf{L}$, and the action of ${\rm GL}(P(N))$
on $T$ is free. There exists a uniform geometric quotient
$$\phi: T\to R(\sO_{X_S},\xi, P):=T//{\rm GL}(P(N))$$
uniformly corepresents the functor which associates to any $S'\to S$ the set of pairs $(\sE,\beta)$ where $\sE$ is a $p$-semistable torsion free sheaf of
Hilbert polynomial $P$ on $X_{S'}/S'$ satisfying condition ${\rm LF}(\xi)$, and $$\beta:\xi^*(\sE)\cong \sO_{S'}^{\oplus r}$$ is a frame. Moreover, we have the following properties:
\begin{itemize}
\item [(1)] Every point of $R(\sO_{X_S},\xi, P)$ is GIT semistable under the action of ${\rm GL}(r)$ (respect to a $\sL$ obtained from $\textbf{L}$) and the quotient $R(\sO_{X_S},\xi, P)//{\rm GL}(r)$ is naturally equal to
$M^{{\rm LF}(\xi)}(\sO_{X_S},P)$;

\item [(2)] For a geometric point $\alpha=(V,\beta)\in  R(\sO_{X_S},\xi, P)$, the orbit $O(\alpha)$ of $\alpha=(V,\beta)$ under ${\rm GL}(r)$ is closed if and only if $V$ is
a direct sum of $p$-stable sheaves, and $\alpha=(V,\beta)$ is a properly stable point if and only if $V$ is a $p$-stable sheaf.

\end{itemize}
\end{thm}

\begin{proof} The proof is the same with Simpson's proof in characteristic zero. For conveniences of readers, we repeat his proof here, which works in characteristic $p>0$.

The projection $\pi: T\to  Q^{{\rm LF}(\xi)}$ is an affine map and all points of  $Q^{{\rm LF}(\xi)}$ are semistable for the action of ${\rm GL}(P(N))$ respect to linearized line
bundle $\sL_m^b$. Thus if $q\in T$ is any point, then there is an ${\rm GL}(P(N))$-invariant section $\sigma\in {\rm H}^0(Q^{{\rm LF}(\xi)}, \sL_m^{ab})$ such that
$(Q^{{\rm LF}(\xi)})_{\sigma\neq 0}$ is affine and $\sigma(\pi(q))\neq 0$. Then $\pi^*(\sigma)\in {\rm H}^0(T, \mathbf{L}^a)$ is ${\rm GL}(P(N))$-invariant such that
$\pi^*(\sigma)(q)\neq 0$ and $T_{\pi^*(\sigma)\neq 0}=\pi^{-1}((Q^{{\rm LF}(\xi)})_{\sigma\neq 0})$ is affine. Thus any point $q\in T$ is semistable. To prove that every point
of $T$ is stable, the key is a lemma of Simpson (Lemma 4.9 of \cite{Si}), which implies that the stabilizer of any point of $T$ is finite and in particular orbits of all points
of $T$ have same dimension. Thus the orbit of any point of $T$ is closed since no orbit can be contained in the closure of another orbit.

To show the action of ${\rm GL}(P(N))$ on $T$ is free, we must show that
\ga{2.2}{{\rm GL}(P(N))\times T\to T\times_{R(\sO_{X_S},\xi, P)}T,\quad (g,q)\mapsto (g(q),q)}
is a closed immersion. By a general result of Mumford (Corollary 2.5 of \cite{Mu} at page 55), the above morphism \eqref{2.2} is proper. By using again Lemma 4.9 of \cite{Si}, Simpson was
able to show that \eqref{2.2} is an inclusion of functors. A proper map which is an inclusion of functors is a closed immersion. Thus the action of ${\rm GL}(P(N))$ on $T$ is free, which implies that $\phi: T\to R(\sO_{X_S},\xi, P)$ is a principal ${\rm GL}(P(N))$-bundle over $R(\sO_{X_S},\xi, P)$ by Proposition 0.9 of \cite{Mu}.

To prove (1), let $\sL$ be an ample line bundle on $R(\sO_{X_S},\xi, P)$ such that $\phi^*(\sL)=\textbf{L}=\pi^*\sL_m^b$, we prove that for any point
$q\in T$ the point $\phi(q)\in R(\sO_{X_S},\xi, P)$ is GIT semistable under the action of ${\rm GL}(r)$ (respect to a $\sL$). As above, there is an
$\sigma\in {\rm H}^0(Q^{{\rm LF}(\xi)}, \sL_m^{ab})$ such that
$(Q^{{\rm LF}(\xi)})_{\sigma\neq 0}$ is affine and $\sigma(\pi(q))\neq 0$. Then $\pi^*(\sigma)\in {\rm H}^0(T, \mathbf{L}^a)$ is ${\rm GL}(P(N))$-invariant  such that
$\pi^*(\sigma)(q)\neq 0$ and $T_{\pi^*(\sigma)\neq 0}$ is affine. Let $\tau\in {\rm H}^0(R(\sO_{X_S},\xi, P), \sL^a)$ be the section such that $\phi^*(\tau)=\pi^*\sigma$. Then $\tau(\phi(q))\neq 0$ and $R(\sO_{X_S},\xi, P)_{\tau\neq 0}=T_{\pi^*(\sigma)\neq 0}//{\rm GL}(P(N))$ is affine. On the other hand, $\pi^*(\sigma)$ is ${\rm GL}(r)$-invariant since $\pi:T\to Q^{{\rm LF}(\xi)}$ is a principal ${\rm GL}(r)$-bundle, which implies that $\tau$ is ${\rm GL}(r)$-invariant and $\phi(q)$ is semistable under the action of ${\rm GL}(r)$ respect to $\sL$. Let
$$\psi: R(\sO_{X_S},\xi, P)\to \mathbf{M}:=R(\sO_{X_S},\xi, P)//{\rm GL}(r).$$
Then both $T\xrightarrow{\phi}R(\sO_{X_S},\xi, P)\xrightarrow{\psi} \mathbf{M}$ and
$$T\xrightarrow{\pi}Q^{{\rm LF}(\xi)}\xrightarrow{\varphi}M^{{\rm LF}(\xi)}(\sO_{X_S},P)$$
are categorical quotients of $T$ by ${\rm GL}(P(N))\times {\rm GL}(r)$. Thus $\mathbf{M}$ is naturally equal to
$M^{{\rm LF}(\xi)}(\sO_{X_S},P)$ and we the commutative diagram
$$\CD
  T @>\phi>> {R(\sO_{X_S},\xi, P)} \\
  @V \pi VV @V \psi VV  \\
  Q^{{\rm LF}(\xi)} @>\varphi>> M^{{\rm LF}(\xi)}(\sO_{X_S},P).
\endCD $$

To prove (2) of the theorem, let $$q=(\sO_{X_{\bar s}}(-N)^{\oplus P(N)}\xrightarrow{q} V\to 0, \beta)\in T$$ such that
$\phi(q)=\alpha=(V,\beta)\in  R(\sO_{X_S},\xi, P)$ and
$$q':=\pi(q)=(\sO_{X_{\bar s}}(-N)^{\oplus P(N)}\xrightarrow{q} V\to 0)\in Q^{{\rm LF}(\xi)}.$$
Let $O_{{\rm GL}(r)}(\alpha)\subset R(\sO_{X_S},\xi, P)$ (resp. $O_{{\rm GL}(P(N))}(q')\subset Q^{{\rm LF}(\xi)}$) be the orbit of $\alpha$ (resp. $q'$)
under ${\rm GL}(r)$ (resp. ${\rm GL}(P(N))$). Then
$$\phi^{-1}(O_{{\rm GL}(r)}(\alpha))=O_{{\rm GL}(P(N))\times{\rm GL}(r)}(q)=\pi^{-1}(O_{{\rm GL}(P(N))}(q'))$$
since $T\xrightarrow{\phi}R(\sO_{X_S},\xi, P)$ (resp. $T\xrightarrow{\pi}Q^{{\rm LF}(\xi)}$) is a principal ${\rm GL}(P(N))$-bundle (resp. a principal ${\rm GL}(r)$-bundle),
where $$O_{{\rm GL}(P(N))\times{\rm GL}(r)}(q)\subset T$$ is the orbit of $q\in T$ under ${\rm GL}(P(N))\times{\rm GL}(r)$. Thus
$$O_{{\rm GL}(r)}(\alpha)\subset R(\sO_{X_S},\xi, P)$$
is closed if and only if $O_{{\rm GL}(P(N))}(q')\subset Q^{{\rm LF}(\xi)}$ is closed. But
$$O_{{\rm GL}(P(N))}(q')\subset Q^{{\rm LF}(\xi)}$$ is closed if and only if $V$ is a direct sum of $p$-stable sheaves. On the other hand, the group of automorphisms of
determinant one of such a direct sum is finite if and only if the sum has exactly one $p$-stable component. Thus $\alpha=(V,\beta)\in R(\sO_{X_S},\xi, P)$ is a properly
stable point if and only if $V$ is a $p$-stable sheaf.
\end{proof}

\begin{rmks}\label{rmk2.4} (1) According to Simpson, the moduli spaces $$R(\sO_{\sX},\xi, P)\to S$$ are called \textbf{Representation spaces}.
In his Theorem 4.10 of \cite{Si}, Simpson proved in fact that $R(\sO_{X_S},\xi, P)$ is a fine moduli space and represents the functor
in the theorem, which we do not pursue since we do not need it for this paper (but it might be also true in the case of characteristic $p>0$).

(2) Let $X$ be a smooth, connected projective variety over a perfect field $k$ of characteristic $p>0$ with a given point $$\xi:{\rm Spec}(k)\to X,$$
and $X_S\to S$ be a smooth projective model of $X\to {\rm Spec}(k)$ such that $\xi:{\rm Spec}(k)\to X$ extends to a section $\xi_S:S\to X_S$. Then the representation spaces
$$R(\sO_{X_S},\xi_S, P)\to S,$$
we will use in this article, are the cases when $P(m)=\chi(\sO_X(m)^{\oplus r})$.

\end{rmks}

\begin{prop}\label{prop2.5} There exists a rational map $$f: R(\sO_{X_S},\xi_S, P)\dashrightarrow R(\sO_{X_S},\xi_S, P)$$ over
$S$ satisfying the following conditions:\begin{itemize} \item [(1)] $\forall\,\alpha=(E,\beta)\in R(\sO_X,\xi, P)=R(\sO_{X_S},\xi_S, P)\times_S{\rm Spec}(k)$, $f$ is
well-defined at $\alpha$ if and only if $F^*_XE$ is $p$-semistable, where $F_X:X\to X$ is the (absolute) Frobenius map. In this case,
$$f(\alpha)=(F^*_XE, F^*_k\beta).$$
\item [(2)] $\forall$ geometric closed point $\alpha_{\bar s}=(\sE_{\bar s},\beta_{\bar s})\in R(\sO_{X_S},\xi_S, P)$, $f$ is well-defined at $\alpha_{\bar s}$ if and only if $F^*_{X_{\bar s}}\sE_{\bar s}$ is $p$-semistable, where $X_{\bar s}$ is a geometric closed fiber of $X_S\to S$. In this case,
$$f(\alpha_{\bar s})=(F^*_{X_{\bar s}}\sE_{\bar s}, F^*_{k(\bar s)}\beta_{\bar s}).$$
\end{itemize}
\end{prop}

\begin{proof} Let $(\sE^{univ},\beta^{univ})$ be the universal object on $X_T:=X_S\times_ST$ where $\beta^{univ}:\xi_T^*\sE^{univ}\cong \sO_T^{\oplus r}$
is the universal frame of $\xi_T^*\sE^{univ}$. Let $$F:X_T\to X_T'$$ denote the relative
Frobenius morphism over $T$. Consider
\ga{2.3}{\xymatrix{\ar@/^20pt/[rr]^{F_{X_T}} X_T\ar[r]^F \ar[dr]_{p_T} & X_T'\ar[r]\ar[d]^{p'_T}
& X_T\ar[d]^{p_T}\\
 & T\ar[r]^{F_T}& T.} }
Then the pullback $(F^*_{X_T}\sE^{univ}, F^*_{T}\beta^{univ})$ of $(\sE^{univ},\beta^{univ})$, where
$$F^*_{T}\beta^{univ}: \xi_T^*(F^*_{X_T}\sE^{univ})=F_T^*(\xi_T^*\sE^{univ})\xrightarrow{F_T^*(\beta^{univ})} F^*_T( \sO_T^{\oplus r})=\sO_T^{\oplus r},$$
defines a ${\rm GL}(P(N))$-invariant morphism $\hat{f}: T_0\to R(\sO_{X_S},\xi_S, P)$ on an ${\rm GL}(P(N))$-invariant open set $T_0\subset T$, which induces the rational map
\ga{2.4} {f: R(\sO_{X_S},\xi_S, P)\dashrightarrow R(\sO_{X_S},\xi_S, P).}

To see that the rational map \eqref{2.4} satisfies the requirements (1) and (2) in the proposition, it is enough to make the following remak.
For any point $t\to T$, let $s\to S$ be its image under $T\to S$,
and $X_s$ be the fiber of $X_S\to S$ at $s\to S$. Then the diagram
\eqref{2.3} specializes to
$$\xymatrix{\ar@/^20pt/[rr]^{F_{X_s\times t}} X_s\times
t\ar[r]^{F_t} \ar[dr] & (X_s\times t)'\ar[r]\ar[d]^{}
& X_s\times t\ar[d]\\
 & t\ar[r]^{}& t.} $$

\end{proof}

\section{Stratified bundles and Representation spaces}

Let $k$ be a perfect field of characteristic $p>0$,
and $X$ a smooth connected projective variety over $k$. A stratified
bundle (or $\sD$-module) on $X$ is by definition a coherent $\sO_X$-module $\sE$ with
a homomorphism $$\nabla: \sD_X\to \sE nd_k(\sE)$$ of
$\sO_X$-algebras, where $\sD_X$ is the sheaf of differential
operators acting on the structure sheaf of $X$. By a theorem of Katz
(cf. \cite[Theorem 1.3]{Gi}), it is equivalent to the following
definition.
\begin{defn}\label{defn3.1}
A stratified bundle on $X$ is a sequence of bundles $$E=\{E_0, E_1,
E_2, \cdots, \sigma_0, \sigma_1,\ldots \}=\{E_i,\sigma_i\}_{i\in\Bbb
N}$$ where $\sigma_i:F_X^*E_{i+1}\to E_i$  is a $\sO_X$-linear
isomorphism, and  $F_X:X\to X$ is the absolute Frobenius.
\end{defn}
A morphism $\alpha=\{\alpha_i\}:\{E_i,\sigma_i\}\to \{F_i,\tau_i\}$
between two stratified bundles is a sequence of morphisms $\alpha_i:
E_i\to F_i$ of $\sO_X$-modules such that
$$\xymatrix{
   F_X^*E_{i+1} \ar[d]_{\sigma_i} \ar[r]^{ F_X^*\alpha_{i+1}}
                &  \ar[d]^{\tau_i}  F_X^*F_{i+1}\\
  E_i  \ar[r]^{\alpha_i}
                &             F_i }$$
is commutative. The category $\textbf{str}(X)$ of stratified bundles
is abelian, rigid, monoidal. We will drop the isomorphisms $\sigma_i$ and will use the notation $E=(E_i)_{i\in\Bbb N}$ and in particular
$E(n_0)=(E_i)_{i\ge n_0}$ is also a stratified bundle.

\begin{lem}\label{lem3.2} Let $E=(E_i)_{i\in\mathbb{N}}$ be a stratified bundle. Then there is an $n_0$ such that
$E_i$ are $p$-semistable of $p(E_i,m)=p(\sO_X,m)$ for all $i\ge n_0$.
\end{lem}

\begin{proof} It is known that $p(E_i,m)=p(\sO_X,m)$ for all $i\ge 0$ (see Corollary 2.2 of \cite{EM}). By Proposition 2.3 of \cite{EM}, there is an
$n_0>0$ such that the stratified bundle $E(n_0)=(E_i)_{i\ge n_0}$ is a successive extension of stratified bundles $U=(U_i)_{i\in\mathbb{N}}$
with the property that all $U_i$ for $i\in \mathbb{N}$ are $p$-stable bundles with $p(U_i,m)=p(\sO_X,m)$. Then, by Lemma \ref{lem2.1}, all $E_i$ for
$i\ge n_0$ are $p$-semistable of $p(E_i,m)=p(\sO_X,m)$.
\end{proof}

\begin{lem}\label{lem3.3} Let $X$ be a smooth projective variety over a perfect field $k$ of characteristic $p>0$ with a fixed rational point $\xi: {\rm Spec}(k)\to X$.
Let $F_X:X\to X$ be the Frobenius map, and $V$, $V'$ be vector bundles satisfying $V=F_X^*(V')$. Then, for any frame $\beta:\xi^*V\cong \sO_{{\rm Spec}(k)}^{\oplus r}$,
there is a unique frame $\beta':\xi^*V'\cong \sO_{{\rm Spec}(k)}^{\oplus r}$ such that $\beta=F_k^*(\beta')$ where $F_k: {\rm Spec}(k)\to {\rm Spec}(k)$ is the Frobenius morphism and $$F^*_k(\beta'): \xi^*V=F_k^*(\xi^*V')\xrightarrow{F_k^*(\beta')} F_k^*(\sO_{{\rm Spec}(k)}^{\oplus r})=\sO_{{\rm Spec}(k)}^{\oplus r}.$$
\end{lem}

\begin{proof} It is clearly a local question and we can assume $X={\rm Spec}(A)$, $\xi^*V=V\otimes_Ak=V/mV$, $\xi^*V'=V'\otimes_Ak=V'/mV'$ (where $k=A/m$, $m\subset A$ is a maximal ideal) and $\xi^*V=F^*_k\xi^*V'=\xi^*V'\otimes_{k^p}k$. The frame $\beta$ is uniquely determined by a base $\beta_1,\,\cdots,\,\beta_r \in \xi^*V$ of $\xi^*V$. Since $k$ is perfect, there are uniquely $\beta'_1,\,\cdots,\,\beta'_r\in \xi^*V'$ such that $\beta_i=\beta_i'\otimes_{k^p}1$ ($1\le i\le r$). Then $\beta'_1,\,\cdots,\,\beta'_r\in \xi^*V'$ are clearly linear independent (thus it is a base of $\xi^*V'$), which defines a frame $\beta':\xi^*V'\cong \sO_{{\rm Spec}(k)}^{\oplus r}$ satisfying
$F_k^*(\beta')=\beta$.
\end{proof}

For a given stratified bundle $E=(E_i)_{i\in \mathbb{N}}$ on $X$, we can fix a frame $$\beta_1:\xi^*E_1\cong \sO_{{\rm Spec}(k)}^{\oplus r}.$$
By Lemma \ref{lem3.2} and Lemma \ref{lem3.3}, we get a set of $k$-points
\ga{3.1} {R(E)_{n_0}:=\{\alpha_i=(E_i,\beta_i)\}_{i\ge n_0}\subset R(\sO_{X},\xi, P)(k)}
with $F^*_X(\alpha_{i+1})=\alpha_i$. To produce a representation of $\pi_1=\pi_1^{\text{{\'e}t}}(X, a)$ from the stratified bundle $E$, the following two results are important.

\begin{lem}[Lange-Stuhler, \cite{LS}]\label{lem3.4}
Let $X$ be a smooth projective variety over $k$ of ${\rm
char}(k)=p>0$, $F_X: X\to X$ be the Frobenius map. If there is a
vector bundle $\mathcal{E}$ on $X$ and an integer $m>0$ such that
$$(F_X^m)^*\mathcal{E}\cong \mathcal{E}$$ then there exists a
geometrically connected {\'e}tale finite cover $$\sigma:Z\to X$$ such
that $\sigma^*\mathcal{E}\cong \mathcal{O}_Z^{\oplus {\rm
rk}(\mathcal{E})}$. This gives a representation
$$\pi_1^{\text{{\'e}t}}(X\otimes_k\bar k, a)\to GL(V)$$
whose associated bundle is $\mathcal{E}\otimes_k\bar k$ on
$X\otimes_k\bar k$.
\end{lem}

To find Frobenius periodic bundles from a given stratified bundle, the key tool is a theorem of
Hrushovski. In fact, we only need a special case of his theorem
\cite[Corollary~1.2]{H} (see also \cite{Va} for a proof in algebraic geometry).

\begin{thm}[Corollary of twisted Lang-Weil estimate, {\cite[Corollary~1.2]{H}}]\label{thm3.5}
Let $Y$ be an affine variety over $\mathbb{F}_q$, and let
$$\Gamma\subset (Y\times_{\F_q} Y)\otimes_{\F_q} \bar \F_q$$ be an
irreducible subvariety over $\overline{\mathbb{F}}_q$. Assume the
two projections $\Gamma\to Y$ are dominant. Then, for any closed
subvariety $W\varsubsetneq Y$, there exists $x\in
Y(\overline{\mathbb{F}}_q)$ such that $(x, x^{q^m})\in \Gamma$ and
$x\notin W$ for large enough natural number $m$.
\end{thm}

However, we will use Hrushovski's theorem in following formulation.

\begin{lem}[Corollary of Hrushovski's theorem]\label{lem3.6}
Let $Y$ be a variety over $\bar{\mathbb{F}}_p$ and $f:
Y\dashrightarrow Y$ a dominant rational map. Then the subset
$\{\,x\in Y(\bar{\mathbb{F}}_p)\,\,|\,\,\exists\,
b,\,\,f^b(x)=x\,\}\subset Y(\bar{\mathbb{F}}_p)$ is dense in $Y$.
\end{lem}

\begin{proof} We prove that any nontrivial affine open set of $Y$ contains a periodic point of $f$.
Replace $Y$ by the affine open set, we can assume that $Y\subset \mathbb{A}^n_{\mathbb{F}_q}$ is an affine variety over $\mathbb{F}_q$ and
$f:Y\dashrightarrow Y$ is also defined over $\mathbb{F}_q$. Let
$\Gamma=\overline{graph(f)}\subset Y\times Y$ and $W\subset Y$ where $f$ is not well-defined. Let $f=(f_1,...,f_n)$ be defined by the rational functions $ f_i\in
\mathbb{F}_q(Y)$ on $Y$. Then, by Theorem \ref{thm3.5}, there is a point $ x=(x_1,...,x_n)\in Y(\overline{\mathbb{F}}_q)$ such
that $x\notin W$ and $(x,x^{q^m})\in\Gamma$ where $$x^{q^m}:=(x_1^{q^m},x_2^{q^m},...,x_n^{q^m}),$$
which implies that $f(x)=(x_1^{q^m},...,x_n^{q^m})=x^{q^m}$. On the other hand, since $f$ is defined by rational functions $f_i\in\mathbb{F}_q(Y)$, we have
$$f(x^{q^a})=(f_1(x^{q^a}),...,f_n(x^{q^a}))=(f_1(x)^{q^a},...,f_n(x)^{q^a}):=f(x)^{q^a}$$
for any integer $a>0$. Thus $$f^b(x)=f^{b-1}(x^{q^m})=f^{b-2}(x^{q^{2m}})=\cdots=x^{q^{bm}}=x$$
when $b$ is large enough since $ x=(x_1,...,x_n)\in Y(\overline{\mathbb{F}}_q)$.
\end{proof}

\begin{thm}\label{thm3.7} Let $X$ be a smooth projective variety over a perfect field $k$ of characteristic $p>0$ with a fixed point
$\xi:{\rm Spec}(k)\to X.$
Let $E=(E_i)_{i\in \mathbb{N}}$ be a stratified bundle on $X$ such that $$\Sigma(E)=\{\,E_i\,\}_{i\in \mathbb{N}}$$
is an infinite set. Then there exist a choice of $S$ and an irreducible closed subset $\sN(E)_S\subset R(\sO_{X_S},\xi_S, P)$ such that
\begin{itemize}
\item [(1)] $\{\,V\in\Sigma(E)\,|\,(V,\beta)\in\sN(E)_S(k)\}\subset\Sigma(E)$  is an infinite subset, where
$\sN(E)_S(k)$ denote the set of $k$-points of $\sN(E)_S$.
\item [(2)] $\sN(E)_S$ contains a dense subset of points $\alpha_{\bar s}=(V_{\bar s},\beta_{\bar s})$, where $V_{\bar s}$ are vector bundles on geometric fibers $X_{\bar s}$ associated to representations
$\rho_{\bar s}: \pi_1^{\text{{\'e}t}}(X_{\bar s},\xi_{\bar s})\to GL(r, k(\bar s))$.
\end{itemize}
\end{thm}

\begin{proof} For any $k$-scheme $T$ and a $T$-flat family of $p$-semistable bundles $\sE$ on $X\times T$ of Hilbert polynomial $P(m)=\chi(\sO_X(m)^{\oplus r})$ with frame
$$\beta:\xi_T^*\sE\cong\sO_T^{\oplus r},$$
the Frobenius pullback $(F^*_{X\times T}\sE, F^*_T\beta)$ induces a rational map
$$T\dashrightarrow  R(\sO_X,\xi,P).$$
This way gives a rational map $F^*_X: R(\sO_X,\xi,P)\dashrightarrow R(\sO_X,\xi,P)$ .

For the given stratified bundle $E=(E_i)_{i\in \mathbb{N}}$, without loss of generality, we assume that all bundles $E_i$ ($\in \mathbb{N}$) are $p$-semistable. Fix a frame
$\beta_1: \xi^*E_1\cong \sO_{{\rm Spec}(k)}^{\oplus r}$, by Lemma \ref{lem3.3}, the stratified bundle $E=(E_i)_{i\in \mathbb{N}}$ gives an infinite set of points
$$R(E)=\{\, Q_i=(E_i,\beta_i)\in R(\sO_X,\xi,P)\,|\,F^*_X(Q_{i+1})=Q_i\,\}$$
and subsets $R(E)_n:=\{\,Q_i\in R(E)\,\}_{i\ge n}$ ($n\in \mathbb{N}$), which satisfy
\ga{3.2}{R(E)=R(E)_1\supseteq R(E)_2\supseteq\cdots \supseteq R(E)_n\supseteq R(E)_{n+1}\supseteq\cdots}
and $F^*_X(R(E)_{n+1})=R(E)_{n}$. Let
$\sZ_n=\overline{R(E)_n}\subset R(\sO_{X},\xi,P)$ be the Zariski closure of $R(E)_n$. Then, by \eqref{3.2}, we have
$$\sZ_1\supseteq\sZ_2\supseteq\cdots\supseteq\sZ_n\supseteq\sZ_{n+1}\supseteq\cdots$$
which implies that there is $n_0>0$ such that $\sZ_n=\sZ_{n_0}$ ($n\ge n_0$). Let
$$\sZ=\bigcap_{i=1}^{\infty}\sZ_i\subset R(\sO_{X},\xi,P).$$
Then the rational map
$F_X^*: R(\sO_{X},\xi,P)\dashrightarrow R(\sO_{X},\xi,P)$
induce a dominant rational map $F_X^*:\sZ\dashrightarrow\sZ$. Thus there is an irreducible component $\sN(E)\subset\sZ\subset R(\sO_{X},\xi,P)$
such that \begin{itemize}
\item  $\{\,V\in\Sigma(E)\,|\,(V,\beta)\in\sN(E)(k)\}\subset\Sigma(E)$ is an infinite subset;
\item  there is an integer $a>0$ such that $(F_X^*)^a:\sZ\dashrightarrow\sZ$ induces
a dominant rational map $(F_X^*)^a:\sN(E)\dashrightarrow\sN(E).$\end{itemize}

Choose a smooth, geometrically irreducible affine variety $S$ over a finite field $\mathbb{F}_q$ with rational function field $\mathbb{F}_q(S)\subset k$ such that $\sN(E)$, $(F_X^*)^a$
are defined over $S$ and
there exist a smooth model $X_S\to S$ of $X\to {\rm Spec}(k)$ and a section $\xi_S:S\to X_S$ extending $\xi\in X(k)$.
Let $$R(\sO_{X_S},\xi_S,P)\to S$$ be the representation space constructed in Theorem \ref{thm2.3}. Then
$$R(\sO_X,\xi,P)=R(\sO_{X_S},\xi_S,P)\times_S{\rm Spec}(k)$$
and the subvariety $\sN(E)\subset R(\sO_{X_S},\xi_S,P)\times_S{\rm Spec}(k)$ is defined over $S$ by the choice of $S$.
Thus there exists a closed subvariety
\ga{3.3}{\sN(E)_S\subset R(\sO_{X_S},\xi_S,P)}
such that $\sN(E)=\sN(E)_S\times_S{\rm Spec}(k)$, which implies that
$$\sN(E)_S(k)=\sN(E)(k).$$
This proves (1) of the theorem.

To show statement (2) of the theorem, recall the rational map
$$f: R(\sO_{X_S},\xi_S,P)\dashrightarrow R(\sO_{X_S},\xi_S,P)$$
constructed in Proposition \ref{prop2.5}, in which we proved that $F_X^*=f\otimes k$ is induced by
$f: R(\sO_{X_S},\xi_S,P)\dashrightarrow R(\sO_{X_S},\xi_S,P)$ under the base change
$$R(\sO_{X_S},\xi_S,P)\times_S{\rm Spec}(k)\xrightarrow{f\otimes k}R(\sO_{X_S},\xi_S,P)\times_S{\rm Spec}(k).$$
Then $f^a\otimes k:\sN(E)_S\times_S{\rm Spec}(k)\dashrightarrow\sN(E)_S\times_S{\rm Spec}(k)$ is a dominant rational map,
which implies (by shrinking $S$ if necessary) that $$f^a:\sN(E)_S\dashrightarrow\sN(E)_S$$ is a dominant rational map over $\mathbb{F}_q$. By Lemma \ref{lem3.6}, the subset
$$\Gamma=\{\,\alpha_{\bar s}\in \sN(E)_S(\bar{\mathbb{F}}_q)\,\,|\,\,\exists\,
m,\,\,f^m(\alpha_{\bar s})=\alpha_{\bar s}\,\}$$ of $f^a$-periodic points is dense in $\sN(E)_S$. By Lemma \ref{lem3.4}, if a point $\alpha_{\bar s}=(V_{\bar s},\beta_{\bar s})\in\Gamma$, then $V_{\bar s}$ is a vector bundle on a geometric fiber $X_{\bar s}$, which is associated to a representation
$\rho_{\bar s}: \pi_1^{\text{{\'e}t}}(X_{\bar s},\xi_{\bar s})\to GL(r, k(\bar s))$.

\end{proof}

\section{An application of Representation spaces}

In this section, we present an application of our Theorem \ref{thm3.7} by giving a proof of relative version of Gieseker's problem.
To warm up, we prove firstly the main theorem of \cite{EM} via representation spaces.

\begin{thm}[Esnault-Mehta, {\cite[Theorem~3.15]{EM}}]\label{thm4.1}
Let $X$ be a smooth connected projective variety defined over a perfect field $k$ of characteristic $p>0$ with a $k$-rational point
$\xi\in X(k)$. If $\pi_1^{\text{{\'e}t}}(X_{\bar k},\xi)=\{1\}$, there is no nontrivial stratified bundle on $X$.
\end{thm}

\begin{proof} We prove it by contradiction. If there is a nontrivial stratified bundle $E=(E_i)_{i\in \mathbb{N}}$ on $X$, without loss of generality, we assume
that all $E_i\in\Sigma(E)=\{E_i\}_{i\in \mathbb{N}}$ are nontrivial bundles. If $\Sigma(E)$ is a finite set, there is an $E_{i_0}\in\Sigma(E)$ such that
$(F^*_X)^aE_{i_0}=E_{i_0}$ for some integer $a>0$. By Lemma \ref{lem3.4}, there is a nontrivial geometrically connected {\'e}tale finite cover $\sigma:Z\to X$ such
that $\sigma^*E_{i_0}\cong \mathcal{O}_Z^{\oplus r}$, which is a contradiction with $\pi_1^{\text{{\'e}t}}(X_{\bar k},\xi)=\{1\}$.

If $E_i\in\Sigma(E)=\{E_i\}_{i\in \mathbb{N}}$ is an infinite set, let $\sN(E)_S\subset R(\sO_{X_S},\xi_S, P)$ be the closed subset constructed in Theorem \ref{thm3.7} and
$$\mathbf{B}=\{\,(V,\beta\,)\in \sN(E)_S\,|\, \text{$V$ is trivial}\,\}\subset R(\sO_{X_S},\xi_S, P).$$
By (2) of Theorem \ref{thm2.3}, $\mathbf{B}$ is a closed subset. By the assumption that $E=(E_i)_{i\in \mathbb{N}}$ is a nontrivial stratified bundle, the open set
$$U=\sN(E)_S\setminus\mathbf{B}$$ is non-empty. Then, by (2) of Theorem \ref{thm3.7}, there is a point $$\alpha_{\bar s}=(V_{\bar s},\beta_{\bar s})\in U$$
such that $V_{\bar s}$ is a vector bundle on a geometric fiber $X_{\bar s}$ associated to a representation
$\rho_{\bar s}: \pi_1^{\text{{\'e}t}}(X_{\bar s},\xi_{\bar s})\to GL(r, k(\bar s))$, which must be nontrivial by definition of $U$. This is a contradiction with
$\pi_1^{\text{{\'e}t}}(X_{\bar k},\xi)=\{1\}$ since the specialization
homomorphism $\pi_1^{\text{{\'e}t}}(X_{\bar k},\xi)\to \pi_1^{\text{{\'e}t}}(X_{\bar s},\xi_{\bar s})$ is surjective (\cite[Expos\'e~X,~Th\'eor\`eme~3.8]{SGA1}).
\end{proof}

\begin{thm}\label{thm4.2} Let $f:Y\to X$ be a morphism of smooth projective
varieties over a perfect field $k$ of characteristic $p>0$, $\xi'\in Y(k)$ and $\xi\in X(k)$ be $k$-points such that $f(\xi')=\xi$.
If the homomorphism
$$f_*: \pi_1^{{\rm \acute{e}t}}(Y_{\bar k},\xi')\to
\pi_1^{{\rm \acute{e}t}}(X_{\bar k},\xi)$$ is trivial, then for any stratified bundle $E$ on $X$, $f^*E$ is trivial.
\end{thm}

\begin{proof} We prove the theorem by contradiction. If there is a stratified bundle $E=(E_i)_{i\in \mathbb{N}}$ on $X$ such that $f^*E$ is nontrivial, without loss of generality, we assume
that all $f^*E_i\in\Sigma(f^*E)=\{f^*E_i\}_{i\in \mathbb{N}}$ are nontrivial bundles.

If $\Sigma(E)=\{E_i\}_{i\in \mathbb{N}}$ is a finite set, there is an $E_{i_0}\in\Sigma(E)$ such that for any $j> i_0$ there is a $1\le j_0\le i_0$ such that $E_{j}=E_{j_0}$ that implies
$$(F^*_X)^{j-j_0}E_j=E_{j_0}=E_j.$$
Thus $E_j$ is induced by a representation of $\pi_1^{{\rm \acute{e}t}}(X_{\bar k},\xi)$ by Lemma \ref{lem3.4}. Then $f^*E_{j}$ is trivial, which implies that all $f^*E_i$ are trivial, a contradiction with our assumption.

If $\Sigma(E)=\{E_i\}_{i\in \mathbb{N}}$ is an infinite set, without loss of generality, we assume
that all $f^*E_i\in\Sigma(f^*E)=\{f^*E_i\}_{i\in \mathbb{N}}$ are $p$-semistable bundles on $Y$ of Hilbert polynomial $P'$.
Let $\sN(E)_S\subset R(\sO_{X_S},\xi_S, P)$ be the closed subset constructed in Theorem \ref{thm3.7} and
$$f_S^*: R(\sO_{X_S},\xi_S, P)\dashrightarrow R(\sO_{Y_S},\xi'_S, P')$$
be the rational map that sends a point $\alpha_{\bar s}=(V_{\bar s},\beta_{\bar s})\in R(\sO_{X_S},\xi_S, P)$ to a point
$f^*_S(\alpha_{\bar s})=(f^*_{\bar s}(V_{\bar s}), f^*_{\bar s}(\beta_{\bar s}))\in R(\sO_{Y_S},\xi'_S, P')$ when $f^*_{\bar s}(V_{\bar s})$ is $p$-semistable on $Y_{\bar s}$,
where $f_S: Y_S\to X_S$ is a model of $f:Y\to X$ and $f_{\bar s}=f_S\otimes k(\bar s): Y_{\bar s}=Y_S\otimes k(\bar s) \to X_S\otimes k(\bar s)= X_{\bar s}$ is the induced morphism on geometric fibers. Consider the open set
$$\mathbf{\sU}=\{\,(V,\beta\,)\in R(\sO_{Y_S},\xi'_S, P')\,|\, \text{$V$ is not trivial}\,\}\subset R(\sO_{Y_S},\xi'_S, P'),$$
which is open by (2) of Theorem \ref{thm2.3}, we have a rational map
\ga{4.1} {f_S^*: \sN(E)_S\dashrightarrow \mathbf{\sU}.}
Let $W\subset\sN(E)_S$ be the open set where the rational map \eqref{4.1} is well defined. Then $W$ contains the infinite set
$$\{\,V\in\Sigma(E)\,|\,(V,\beta)\in\sN(E)_S(k)\}\subset\Sigma(E)$$ defined in (1) of Theorem \ref{thm3.7} since all $f^*E_i$ ($\forall\,E_i\in\Sigma(E)$) are nontrivial $p$-semistable bundles on $Y$. But, by (2) of Theorem \ref{thm3.7}, $W$ contains a point $\alpha_{\bar s}=(V_{\bar s},\beta_{\bar s})$
where $V_{\bar s}$ is a vector bundle on a geometric fiber $X_{\bar s}$ associated to a representation
$\rho_{\bar s}: \pi_1^{\text{{\'e}t}}(X_{\bar s},\xi_{\bar s})\to GL(r, k(\bar s))$. Then $f^*_{\bar s}(V_{\bar s})$ must be trivial by the condition of the
theorem. Thus the rational map \eqref{4.1} is not well-defined at $\alpha_{\bar s}=(V_{\bar s},\beta_{\bar s})\in W$, which is a contradiction.
\end{proof}

\bibliographystyle{plain}

\renewcommand\refname{References}

\end{document}